\documentclass{amsart}

\usepackage{mathtools,amssymb,amsthm,bbm}
\usepackage{cite}

\usepackage[a4paper]{geometry}
\usepackage{enumerate}
\usepackage[colorlinks,final,plainpages=false,pdfpagelabels]{hyperref}

\theoremstyle{plain}
\newtheorem{theorem}{Theorem}
\newtheorem{proposition}[theorem]{Proposition}
\newtheorem{lemma}[theorem]{Lemma}
\newtheorem{corollary}[theorem]{Corollary}

\newtheorem{remark}[theorem]{Remark}
\newtheorem{definition}[theorem]{Definition}
\newcommand{\noopsort}[1]{}

\begin{document}
	\title{Linear Stochastic Dyadic model}
	\author{Luigi Amedeo Bianchi}  
	\address{Luigi Amedeo Bianchi\\
		 Dipartimento di Matematica\\ 
		 Università degli Studi di Trento\\ 
		 Via Sommarive 14\\ 
		 38123 Trento\\ 
		 Italy} 
	 \email{luigiamedeo.bianchi@unitn.it}
	
	\author{Francesco Morandin}
	\address{Francesco Morandin\\
		 Dipartimento di Scienze Matematiche, Fisiche e Informatiche\\ 
		 Università degli Studi di Parma\\ 
		 Parco Area delle Scienze 53/A\\ 
		 43124 Parma\\ 
		 Italy}
		\email{francesco.morandin@unipr.it}

	\begin{abstract}
		We discuss a stochastic interacting particles' system connected to dyadic models of turbulence, defining suitable classes of solutions and proving their existence and uniqueness. We investigate the regularity of a particular family of solutions, called moderate, and we conclude with existence and uniqueness of invariant measures associated with such moderate solutions.
	\end{abstract}
	\thanks{This research was partially supported by Gruppo Nazionale Analisi Matematica, Probabilità e loro Applicazioni (GNAMPA) of Istituto Nazionale di Alta Matematica (INdAM) through the project ``Stationary inverse cascades in shell models of turbulence''. LAB would like to thank the Hausdorff Institute for Mathematics in Bonn, where part of the research was conducted during the Junior Trimester Program ``Randomness, PDEs and Nonlinear Fluctuations''.}
	
	\maketitle	
	\section{Introduction}
	In this paper we consider a 
	stochastic system of interacting particles,
	introduced and discussed in~\cite{BrzFlaNekZeg}:
	\begin{equation}\label{e:main_model_intro}
	\begin{cases}
	dX_n=k_{n-1}X_{n-1}\circ dW_{n-1}-k_nX_{n+1}\circ dW_n
	,& n\geq1\\
	X_n(0)=\overline X_n
	,& n\geq1\\
	X_0(t)\equiv\sigma
	,& t\geq0,
	\end{cases}
	\end{equation}
	where  $k_n:=\lambda^n$, with $\lambda>1$, the $W_n$ are independent Brownian motions with $\circ\, dW$ denoting 
	Stratonovich stochastic integration,
	$\overline{X}$ is a random initial condition, and $\sigma$ is a nonnegative deterministic forcing, a term not present in~\cite{BrzFlaNekZeg}.
	
	It is closely related to dyadic models of turbulence, 
	an interesting simplification of the energy cascade phenomenon, which have been extensively studied in the physical literature, as well as the mathematical one. We mention here just some results: in~\cite{BarFlaMor2010PAMS, BarFlaMor2011AAP, BarMor2013NON, Bianchi2013} one can find dyadic models of a form similar to the one here, linearised by means of Girsanov's theorem, where~\cite{KatPav04, FriPav04, BarMorRom2011, BarBiaFlaMor2013, BiaMor2017CMP} deal with other variants of the dyadic models. We refer to the review papers~\cite{AleBif2018, BiaFla2020}, too, for further reading and references.
	
	The $\sigma$ in~\eqref{e:main_model_intro} is a deterministic forcing, a feature that~\eqref{e:main_model_intro} shares with other dyadic models, see for example~\cite{CheFriPav2007, CheFri09}. Such forcing is usually introduced to provide a steady flow of energy and allow for solutions that are stationary in time: dyadic models, even when formally energy preserving, dissipate energy, through the so-called anomalous dissipation.
	
	In some cases, one considers dyadic models with additive stochastic forcing,
	for example~\cite{AndBarColForPro2016N,FriGlaVic2016AIHPPS}, where the
	noise is 1-dimensional, and~\cite{Rom2013}, where it acts on all components.
	The model~\eqref{e:main_model_intro} considered in this paper is itself stochastic,
	through the random initial
	condition $\overline{X}$, and the infinite-dimensional multiplicative noise $(W_n)_n$, which
	formally conserves energy.  Other examples linked to this one in the
	literature are the already mentioned~\cite{BarFlaMor2010PAMS, BarFlaMor2011AAP, Bianchi2013,
		BarMor2013NON} and~\cite{BrzFlaNekZeg}.
	
	Unlike classical dyadic models, the particle system~\eqref{e:main_model_intro} considered in this paper is linear. This is not surprising: as already mentioned, stochastic linear dyadic models arise from nonlinear ones through Girsanov's theorem. Moreover, the coefficients in such models are growing exponentially, and the associated operator, though linear, is still nontrivial to deal with. Additionally, as mentioned in~\cite{BrzFlaNekZeg}, systems similar to the one discussed 
	here play a role in modelling quantum spin chains and heat conduction.
	
	The main results in this paper are the following: we define two classes of solutions for our system, proper solutions and moderate ones, a more general class. For both classes we prove existence and uniqueness, but the natural setting for the uniqueness is that of moderate solutions, a class that had already been introduced in~\cite{BrzFlaNekZeg}. On the other hand, we prove that, under very mild assumptions on the initial conditions, moderate solutions are much more regular than hinted to by the definition: in particular they have finite energy for positive times. 
	Finally, we move on to invariant measures. By focusing on the more regular solutions suggested by the regularity theorem we just mentioned, we can improve the result in~\cite{BrzFlaNekZeg}, showing that for moderate solutions there exists a unique invariant measure, with support in the space of finite energy solutions.
	
	Before moving on, let us briefly give the general structure of the paper. We begin with the definition of the model and of proper solutions in Section~\ref{sec:solutionsl2}, where we also prove the existence of such solutions. In Section~\ref{sec:moderate} it is the turn of moderate solutions, and their existence and uniqueness. After that, in Section~\ref{sec:markovchain}, we take an apparent detour, considering a related continuous-time Markov chain, which will allow us to better characterize moderate solutions and show that they are quite regular. Finally, in Section~\ref{sec:invariant}, we show existence and uniqueness of invariant measures for our system.

	\section{Model and proper solutions}
	\label{sec:solutionsl2}
	The model studied in this paper is the following linear and formally conservative system of interacting particles, introduced in~\cite{BrzFlaNekZeg}: 
	\begin{equation}\label{e:main_model}
	\begin{cases}
	dX_n=k_{n-1}X_{n-1}\circ dW_{n-1}-k_nX_{n+1}\circ dW_n
	,& n\geq1\\
	X_n(0)=\overline X_n
	,& n\geq1\\
	X_0(t)\equiv\sigma
	,& t\geq0.
	\end{cases}
	\end{equation}
	Here, for $n\geq0$, the coefficients are $k_n:=\lambda^n$, for some
	$\lambda>1$, the $W_n$ are independent Brownian motions on a given
	filtered probability space $(\Omega, \mathcal{F}, \mathcal{F}_t, P)$,
	and $\circ\, dW$ denotes the Stratonovich stochastic integration,
	$\overline{X}$ is an $\mathcal F_0$-mea\-sur\-able random initial
	condition, and $\sigma\geq0$ is a constant and deterministic forcing
	term.
	
	We can rewrite each differential equation in It\=o form:
	\begin{multline*}
	dX_n=
	k_{n-1}X_{n-1}dW_{n-1}
	-k_nX_{n+1}dW_n \\
	+\frac12d[k_{n-1}X_{n-1},W_{n-1}]
	-\frac12d[k_nX_{n+1},W_n],
	\end{multline*}
	where, by~\eqref{e:main_model}
	\[
	d[X_{n-1},W_{n-1}]
	=-k_{n-1}X_ndt
	,\  n\geq2, \ \text{and}\ 
	d[X_{n+1},W_n]
	=k_nX_ndt
	,\  n\geq1,
	\]
	and rewrite~\eqref{e:main_model} in the following way
	\begin{equation}\label{e:main_ito}
	\begin{dcases}
	dX_1=
	\sigma dW_0-k_1X_2dW_1-\frac12k_1^2X_1dt\\
	dX_n
	=k_{n-1}X_{n-1}dW_{n-1}-k_nX_{n+1}dW_n-\frac12(k_{n-1}^2+k_n^2)X_ndt
	,\qquad n\geq2\\
	X(0)=\overline X.
	\end{dcases}
	\end{equation}
	
	We consider as initial condition $\overline X$ an
	$\mathcal F_0$-measurable random variable, as mentioned, which will
	usually take values in some space $H^s$, for $s\in \mathbb{R}$, where
	\[
	H^s:=\biggl\{x\in\mathbb R^{\mathbb N}:\|x\|_{H^s}:=\Bigl(\sum_{n\geq1}k_n^{2s}x_n^2\Bigr)^{1/2}<\infty\biggr\}.
	\]
	
	\begin{remark}
		These spaces have nice properties: they are Hilbert and separable. We also have that $H^s\subseteq H^p$ for $p<s$, and $\|\cdot\|_{H^p}\leq \|\cdot\|_{H^s}$.
		Notice, moreover, that $H^0=l^2$.
                The $l^2$ norm is identified as the energy of the
                configuration and the spaces $H^s$ may be seen as
                corresponding to the usual function spaces
                (see~\cite{BiaMor2017CMP} for a thorough explanation
                in a related model). 
	\end{remark}
	
	We introduce now the definition of proper solutions which
        will be our starting point towards the more general moderate
        solutions, proposed in~\cite{BrzFlaNekZeg}, that here appear in
        Definition~\ref{d:moderate_sol}.
	
	\begin{definition}\label{d:proper_sols}
		Given a filtered probability space
		$(\Omega,\mathcal F,\mathcal F_t,P)$, an $\mathcal F_0$-mea\-sur\-able
		random variable $\overline X$, taking values in some
		$H^s$, and a sequence of independent Brownian motions
		$(W_n)_{n\geq0}$, we say that a process
		$X=(X_n(t))_{n\geq1,t\in[0,T]}$ is a \emph{componentwise solution}
		with initial condition $\overline X$, if it has adapted, continuous
		components and satisfies system~\eqref{e:main_ito}.
		
		If a componentwise solution $X$ is in $L^2([0,T]\times\Omega;l^2)$
		and $X_n\in L^4([0,T]\times\Omega)$ for all $n\geq1$, we say that
		$X$ is a \emph{proper solution}.
	\end{definition}
	
	The requirement
	of finite fourth moments, which appears in the definition of proper
	solution, is a technical assumption needed in Proposition~\ref{p:sol_2ndmom}, which shows that second moments of a proper solution solve a closed system of equations (see also~\cite{BarMor2013NON},~\cite{Bianchi2013}). Fourth moments play also a role in Theorem \ref{t:existence_l2init}.
	
	Let us now state and prove the following existence result for proper solutions with initial conditions with finite energy.
	
	\begin{theorem}\label{t:existence_l2init}
		For any initial condition
		$\overline X\in L^4(\Omega,\mathcal F_0;l^2)$, there exists at least
		a proper solution
		$X\in L^\infty([0,T];L^4(\Omega;l^2))$.  Moreover,
		\begin{equation}\label{e:ps_en_bd_avg}
		E\|X(t)\|^2_{l^2} \leq E\|\overline X\|^2_{l^2}+\sigma^2t
		\qquad t\in[0,T],
		\end{equation}
		and if $\sigma=0$, then with probability 1,
		\begin{equation}\label{e:ps_en_bd_as}
		\|X(t)\|_{l^2}\leq\|\overline X\|_{l^2}
		\qquad t\in[0,T].
		\end{equation}
	\end{theorem}
	
	\begin{proof}
		For $N\geq3$, consider the following SDE in $\mathbb R^N$, which
		represents a Galerkin approximation of the original
		problem~\eqref{e:main_ito}:
		\begin{equation}\label{e:galerkin_sys}
		\begin{dcases}
		dX^{(N)}_1=
		\sigma dW_0-k_1X^{(N)}_2dW_1-\frac12k_1^2X^{(N)}_1dt\\
		dX^{(N)}_n
		=k_{n-1}X^{(N)}_{n-1}dW_{n-1}-k_nX^{(N)}_{n+1}dW_n-\frac12(k_{n-1}^2+k_n^2)X^{(N)}_ndt
		,\quad n=2,\dots,N-1\\
		dX^{(N)}_N
		=k_{N-1}X^{(N)}_{N-1}dW_{N-1}-\frac12k_{N-1}^2X^{(N)}_Ndt\\
		X_n^{(N)}(0)=\overline X_n
		,\quad n=1,\dots,N.
		\end{dcases}
		\end{equation}
		This system has a strong solution with finite fourth moments which we consider embedded in $l^2$ for simplicity. We compute
		\[
		d\|X^{(N)}\|^2_{l^2}
		=\sum_{n\geq1}d\big((X_n^{(N)})^2\big)
		=2\sum_{n\geq1}X_n^{(N)}dX_n^{(N)}+\sum_{n\geq1}d[X_n^{(N)}].
		\]
		Now (dropping the index $^{(N)}$ in the next two equations not to burden the notation too much),
		\[
		2\sum_{n\geq1}X_ndX_n
		=2\sigma X_1dW_0-k_1^2X_1^2dt-\sum_{n=2}^{N-1}(k_{n-1}^2+k_n^2)X_n^2dt-k_{N-1}^2X_N^2dt,
		\]
		and
		\[
		\sum_{n\geq1}d[X_n]
		=\sigma^2dt+k_1^2X_2^2dt+\sum_{n=2}^{N-1}(k_{n-1}^2X_{n-1}^2+k_n^2X_{n+1}^2)dt+k_{N-1}^2X_{N-1}^2dt,
		\]
		hence
		\[
		d\|X^{(N)}\|^2_{l^2}
		=2\sigma X_1^{(N)}dW_0+\sigma^2dt,
		\]
		yielding
		\begin{equation}\label{e:l2_isometry_galerkin}
		\|X^{(N)}(t)\|^2_{l^2}
		=\|X^{(N)}(0)\|^2_{l^2}+2\sigma\int_0^tX_1^{(N)}(s)\mathrm{d}W_0(s)+\sigma^2t
		,\qquad\text{a.s.}
		\end{equation}
		From here we can bound the second moment (with respect to $\Omega$) of
		this $l^2$ norm,
		\begin{equation}\label{e:2nd_mom_bd}
		E\|X^{(N)}(t)\|^2_{l^2}
		\leq E\|\overline X\|^2_{l^2}+\sigma^2t,
		\end{equation}
		and hence by~\eqref{e:l2_isometry_galerkin} again and by It\=o
		isometry, we can also bound the fourth moment,
		\begin{align*}
		E\|X^{(N)}(t)\|^4_{l^2}
		&\leq 3E\|\overline X\|^4_{l^2}+12\sigma^2\int_0^tE\bigl[\bigl(X_1^{(N)}(s)\bigr)^2\bigr]\mathrm{d}s+3\sigma^4t^2\\
		&\leq 3E\|\overline X\|^4_{l^2}+12\sigma^2\int_0^t\bigl(E\|\overline
		X\|^2_{l^2}+\sigma^2s\bigr)\mathrm{d}s+3\sigma^4t^2\\
		&= 3\|\overline X\|^4_{L^4(\Omega;l^2)}+ 12\sigma^2T\|\overline
		X\|^2_{L^2(\Omega;l^2)}+9\sigma^4T^2=:L,
		\end{align*}
		for all $t\in[0,T]$, and $N\geq3$, which we can also write as
		\begin{equation}\label{e:galerkin_bound}
		\|X^{(N)}\|_{L^\infty([0,T];L^4(\Omega;l^2))}\leq L^{1/4}
		,\qquad N\geq3.
		\end{equation}
		
		Consequently the sequence $X^{(N)}$ is bounded in
		$L^\infty([0,T];L^4(\Omega;l^2))$, which is the dual of the space
		$L^1([0,T];L^{4/3}(\Omega;l^2))$. Since the latter is separable (see for example~\cite{HytNeeVerWei2016} for details),
		sequential Banach-Alaoglu theorem applies and there is a
		subsequence $X^{(N_k)}$ which converges in the weak*
		topology to some limit $X^*$ for $k\to\infty$. 
		A fortiori, there is also weak
		convergence in $L^p([0,T]\times\Omega;l^2)$, for all $1<p\leq4$, 
		and in particular for $p=2$.
		
		The components $X^{(N)}_n$ for $n\geq1$ belong to
		$L^2([0,T]\times\Omega;\mathbb R)$ and are progressively
		measurable. The subset of progressively measurable processes is a
		linear subspace of $L^2$ which is complete, hence closed in the strong
		topology. Thus it is closed also in the weak topology. Since
		$X^{(N_k)}_n$ converges to $X^*_n$ in the weak topology of
		$L^2([0,T]\times\Omega;\mathbb R)$, we conclude that $X^*_n$ is
		progressively measurable.
		
		Now we need to pass to the limit
		in~\eqref{e:main_ito}. By~\eqref{e:galerkin_sys} the processes
		$X^{(N)}_n$ satisfy
		\begin{multline*}
		X^{(N)}_n(t)-X^{(N)}_n(0)
		=k_{n-1}\int_0^tX^{(N)}_{n-1}(s)\mathrm{d}W_{n-1}(s)-k_n\int_0^tX^{(N)}_{n+1}(s)\mathrm{d}W_n(s)\\
		-\frac12(k_{n-1}^2+k_n^2)\int_0^tX^{(N)}_n(s)\mathrm{d}s,
		\end{multline*}
		for $N>n$. The maps
		\[
		V\mapsto\int_0^tV(s)\mathrm{d}W_n(s)
		\qquad\text{and}\qquad
		V\mapsto\int_0^tV(s)\mathrm{d}s
		\]
		are linear and (strongly) continuous operators from
		$L^2([0,T]\times\Omega)$ to $L^2(\Omega)$, hence they are weakly
		continuous so we can pass to the limit (see
		Remark~\ref{r:limit_inside_integrals}, below) and conclude that the
		processes $X^*_n$ also satisfy system~\eqref{e:main_ito}.
		
		A posteriori, from these integral equations, it follows that there is
		a modification $X$ of $X^*$ such that all its components are continuous,
		hence $X$ is a componentwise solution in $L^2([0,T]\times\Omega;l^2)$.
		
		To conclude that $X$ is a proper solution, we only need to check that
		the components are in $L^4$. For all $n\geq1$,
		\[
		\|X_n\|_{L^4([0,T]\times\Omega)}^4
		\leq\|X^*\|_{L^4([0,T]\times\Omega;l^2)}^4
		\leq\liminf_{k\to\infty}\|X^{(N_k)}\|_{L^4([0,T]\times\Omega;l^2)}^4
		\leq LT,
		\]
		where the third inequality is a consequence of
		bound~\eqref{e:galerkin_bound} and the second one of the weak lower
		semicontinuity of the norm, i.e.~that in a Banach space, if a sequence
		converges weakly, then the norm of the limit is bounded by the limit inferior
		of the norms.
		
		Then, to prove the bound on energy~\eqref{e:ps_en_bd_avg}, we take a
		measurable set $D\subset[0,T]$, integrate~\eqref{e:2nd_mom_bd} on $D$
		and pass to the limit with the weak lower semicontinuity of the
		$L^2(D\times\Omega;l^2)$ norm, to get
		\[
		\int_D E\|X(t)\|_{l^2}^2 \mathrm{d}t
		\leq\liminf_k\int_D E\|X^{(N_k)}(t)\|_{l^2}^2 \mathrm{d}t
		\leq\int_D \bigl(E\|\overline X\|_{l^2}^2+\sigma^2t\bigr) \mathrm{d}t.
		\]
		By the arbitrariness of $D$, the bound~\eqref{e:ps_en_bd_avg} must hold
		for a.e.~$t$. Now, if it still failed for some $t_0$, then one could
		find $\epsilon>0$ and an integer $m$ such that
		$E\sum_{n\leq m}X_n(t_0)^2-\epsilon$ would also exceed the bound, but
		by the continuity of the trajectories and the finiteness of the sum
		this would give a contradiction.
		
		Finally, to prove the last statement, we follow ideas
		from~\cite{BarMor2013NON}. If $\sigma=0$, by~\eqref{e:l2_isometry_galerkin}, we have
		\begin{equation*}
		\|X^{(N)}\|_{l^2}\leq\|\overline X\|_{l^2}
		\qquad P\text{-a.s.~on all $[0,T]$ and for all $N$.}
		\end{equation*}
		we now integrate the square of this inequality on
		$A:=\{\|X\|_{l^2}>\|\overline X\|_{l^2}\}\subset[0,T]\times\Omega$ and
		pass to the limit with the weak lower semicontinuity of the
		$L^2(A;l^2)$ norm, to get that $A$ must be
		$\mathcal L\otimes P$-negligible. Then for all $m\geq1$ also
		$\{(t,\omega):\sum_{n\leq m}X_n(t)^2>\|\overline X\|_{l^2}^2\}$ is
		negligible, and hence by continuity of trajectories,
		\[
		P\Bigl(\sup_t\sum_{n\leq m}X_n(t)^2
		\leq\|\overline X\|_{l^2}^2\Bigr)
		=1,
		\]
		and we can conclude by intersecting over all $m$. 
	\end{proof}
	
	\begin{remark}
		\label{r:limit_inside_integrals}
		Passing to the limit in the integral equations is standard but made
		somewhat tricky by the different spaces involved, so we expand
		it here for sake of completeness. First of all, we fix $n\geq1$.  We start now from the
		fact that $X_n^{(N)}\to X_n^*$ in weak-$L^2([0,T]\times\Omega)$. For
		$t\in[0,T]$, let $L_{n,t}$ from $L^2([0,T]\times\Omega)$ to $L^2(\Omega)$ be
		formally defined by
		\begin{multline*}
		L_{n,t}(x)
		:=k_{n-1}\int_0^tx_{n-1}(s)\mathrm{d}W_{n-1}(s)-k_n\int_0^tx_{n+1}(s)\mathrm{d}W_n(s)\\-\frac12(k_{n-1}^2+k_n^2)\int_0^tx_n(s)\mathrm{d}s.
		\end{multline*}
		Since the integral operators are weakly continuous,
		$L_{n,t}(X^{(N)})\to L_{n,t}(X^*)$ in weak-$L^2(\Omega)$, for all
		$t\in[0,T]$. On the other hand $X_n^{(N)}(0)\to X^*_n(0)$ a.s.~since by
		construction it is eventually constant. Therefore
		$X_n^{(N)}(t)\to X^*_n(0)+L_{n,t}(X^*)=:Z_t$ in weak-$L^2(\Omega)$,
		for all $t$. It is now enough to strengthen the convergence to
		weak-$L^2([0,T]\times\Omega)$ to conclude that $X_n^*=Z$ and hence
		that it solves the integral equations.
		
		To this end, take any $Y\in L^2([0,T]\times\Omega)$. For a.e.~$t$ we
		have that $Y_t\in L^2(\Omega)$, so by weak convergence
		\[
		g^{(N)}(t):=E[Y_tX_n^{(N)}(t)]\to E[Y_tZ_t]=:h(t).
		\]
		By Cauchy-Schwarz inequality and the uniform bound given
		by~\eqref{e:2nd_mom_bd} we get
		\[
		|g^{(N)}(t)|
		\leq\|Y_t\|_{L^2}\cdot\|X_n^{(N)}(t)\|_{L^2}
		\leq C\|Y_t\|_{L^2}\in L^2(0,T),
		\]
		so by dominated convergence
		\[
		\int_0^TE[Y_tX_n^{(N)}(t)]\mathrm{d}t\to\int_0^TE[Y_tZ_t]\mathrm{d}t,
		\]
		and we are done.
	\end{remark}
	
	We now take a first look at the second moments of a proper solution: they solve a linear system. We will see later on that such property can be used to get useful estimates on the solutions themselves.
	
	\begin{proposition}\label{p:sol_2ndmom}
		Let $X$ be a proper solution with initial condition $\overline
		X$ such that for all $n\geq 1$ $\overline{X}_n\in L^2(\Omega)$. 
		For all $n\geq1$ and $t\in[0,T]$, let $u_n(t):=E[X_n(t)^2]$ and
		$\overline u_n:=E[\overline X_n^2]$.
		
		Then $u\in L^1\bigl([0,T];l^1(\mathbb R_+)\bigr)$ and it satisfies the following linear
		system
		\begin{equation}\label{eq:ode2ndmom}
		\begin{cases}
		u_1'=\sigma^2-k_1^2u_1+k_1^2u_2\\
		u_n'=k_{n-1}^2u_{n-1}-(k_{n-1}^2+k_n^2)u_n+k_n^2u_{n+1}	,\qquad n\geq2,
		\end{cases}
		\end{equation}
		with initial condition $\overline u$.
	\end{proposition}
	
	\begin{proof}
		It follows from~\eqref{e:main_ito}, by applying It\=o formula to
		$X_n^2$, that
		
		\[
		d(X_n^2)
		=2X_ndX_n+d[X_n],
		\]
		where
		\[
		d[X_n]
		=k_{n-1}^2X_{n-1}^2dt+k_n^2X_{n+1}^2dt
		,\qquad n\geq1,
		\]
		hence
		\[
		d(X_1^2)
		=2\sigma X_1dW_0-2k_1X_1X_2dW_1-k_1^2X_1^2dt+\sigma^2dt+k_1^2X_2^2dt,
		\]
		and
		\begin{align*}
		d(X_n^2)
		=&2k_{n-1}X_{n-1}X_ndW_{n-1}-2k_nX_nX_{n+1}dW_n\\
		&-(k_{n-1}^2+k_n^2)X_n^2dt+k_{n-1}^2X_{n-1}^2dt+k_n^2X_{n+1}^2dt
		,\qquad n\geq2.
		\end{align*}
		By the definition of proper solution, $X_n\in L^4([0,T]\times\Omega)$
		for all $n$, so the stochastic integrals above are true martingales,
		and taking expectations and differentiating, we get
		\[
		\begin{dcases}
		\frac d{dt}E[X_1^2]
		=\sigma^2-k_1^2E[X_1^2]+k_1^2E[X_2^2]\\
		\frac d{dt}E[X_n^2]
		=-(k_{n-1}^2+k_n^2)E[X_n^2]+k_{n-1}^2E[X_{n-1}^2]+k_n^2E[X_{n+1}^2]
		,\qquad n\geq2.
		\end{dcases}
		\]
		In other words, the second moments of the components satisfy
		system~\eqref{eq:ode2ndmom}.
		
		The fact that $u(0)=\overline u$ is obvious and $u\in L^1([0,T];l^1)$
		follows from the definition of proper solution, since
		$X\in L^2([0,T]\times\Omega;l^2)$. 
	\end{proof}
	
	There is furthermore a unique constant solution to system~\eqref{eq:ode2ndmom} satisfied by the second moments of proper solutions. This constant solution has an explicit form, as shown in the following result.
	
	\begin{proposition}\label{prop:uniq_stationary_2mom}
          System \eqref{eq:ode2ndmom} has a unique constant solution
          $u(t)\equiv s$ in $l^1(\mathbb R)$, with explicit form
          $s_n=\sigma^2\lambda^{-2n}(1-\lambda^{-2})^{-1}$ for
          $n\geq1$.
          In particular, $s\in H^\beta$ for all $\beta<1$.
	\end{proposition}
	\begin{proof}
		Assume $s=(s_n)_n$ is such a solution. Then
		\begin{equation*}
		\begin{cases}
		\sigma^2-k_1^2s_1+k_1^2s_2 = 0\\
		k_{n-1}^2s_{n-1}-(k_{n-1}^2+k_n^2)s_n+k_n^2s_{n+1} =0
		,\qquad n\geq2.
		\end{cases}
		\end{equation*}
		
		We want to write a recursion for the differences of consecutive elements: we have $s_1-s_2 = \sigma^2 \cdot k_1^{-2}$, and also
		\[
                  s_n-s_{n+1} = k_{n-1}^2\cdot k_n^{-2} (s_{n-1}-s_n)
                  ,\qquad n\geq2.
		\]
		Recall now, that $k_n=\lambda^n = k_1^n$, so $s_n-s_{n+1} = \lambda^{-2} (s_{n-1}-s_n)$, and by recursion
		$s_n-s_{n+1} = \lambda^{-2n}\sigma^2$ for all $n\geq
                1$, yielding that for any $m>1$,
		\[
		s_1-s_m = \sum_{n=1}^m(s_n-s_{n+1}) = \sigma^2 \sum_{n=1}^m \lambda^{-2n} = \sigma^2 \frac{\lambda^{-2}-\lambda^{-2m}}{1-\lambda^{-2}}.
		\]
		In order to have the explicit form of this solution,
                we use now the fact that $s\in l^1$ and so $s_m\to0$:
		\[
		s_1=\sigma^2\frac{\lambda^{-2}-\lambda^{-2m}}{1-\lambda^{-2}}+s_m = \sigma^2\frac{\lambda^{-2}}{1-\lambda^{-2}},
		\]
		and for any $n$, we get
                $s_n=\sigma^2\lambda^{-2n}(1-\lambda^{-2})^{-1}$. It
                is immediate to verify that this is in fact a solution
                in $l^1$.
	\end{proof}
	
	\begin{remark}
          We will prove uniqueness of the solutions of this system in
          Theorem~\ref{t:uniqueness_linear}, in Section~\ref{sec:markovchain}.
           Then Proposition \ref{prop:uniq_stationary_2mom}
          will tell us that in system~\eqref{e:main_ito}, if the
          initial condition is chosen with second moments of the form
          just shown, then proper solutions have higher regularity
          than their definition requires, living in
          $L^\infty([0,T];L^2(\Omega; H^{1^-})$, and their components
          have constant second moments. This suggests the existence of
          invariant measures supported on configurations with
          $H^{1^-}$ regularity, which in fact will be found in Section~\ref{sec:invariant},
          at the end of the paper.
	\end{remark}
	
	\section{Moderate solutions}
	\label{sec:moderate}
	The definition of proper solution given in
        Definition~\ref{d:proper_sols} is in some sense too strong, in
        particular for the assumptions on $\overline X$ in
        Theorem~\ref{t:existence_l2init}, so we would like to consider
        a more general class of solutions. Consequently, we present
        here the concept of \emph{moderate solutions}, as introduced
        in~\cite{BrzFlaNekZeg} to identify a natural space to prove
        existence and uniqueness in, with much weaker requirements on
        initial conditions.  Later, in
        Theorem~\ref{thm:regularity_moderate}, we will show that
        moderate solutions are actually almost as regular as proper
        solutions.
	
	Some of the following results are similar to those in~\cite{BrzFlaNekZeg}, but we include full
	proofs here nevertheless, given that some details differ, that we have an additional forcing term, 
	and for overall completeness.
	
	\begin{definition}\label{d:moderate_sol}
		We say that a 
		componentwise solution
		$X$ is a \emph{moderate solution} with initial
		condition $\overline X$ if: $\overline X\in L^2(\Omega;H^{-1})$,
		$X\in L^2([0,T]\times\Omega;H^{-1})$ and there exists a sequence
		$(X^{(N)})_{N\geq1}$ of proper solutions converging to $X$ in
		$L^2([0,T]\times\Omega;H^{-1})$ as $N\to\infty$, such that their
		initial conditions $\overline X^{(N)}$ converge to $\overline X$ in
		$L^2(\Omega;H^{-1})$.
		
		If a moderate solution is in $L^2([0,T]\times \Omega; l^2)$, we call it a \emph{finite energy (moderate) solution}.
	\end{definition}
	
	\begin{remark}
		Clearly all proper solutions with initial conditions in $L^2(\Omega, H^{-1})$ 
		are finite energy solutions, as can be seen by taking the constant sequence.
	\end{remark}
	
	The key result to prove existence and uniqueness of the moderate
	solution is the following lemma, which has a statement similar to
	Lemma~2.7 from~\cite{BrzFlaNekZeg} and shares the same proof strategy.
	
	\begin{lemma}\label{l:H-1_contraction}	
		If a process $X\in L^2([0,T]\times\Omega;l^2)$ has second moments
		$u_n(t):=E[X_n(t)^2]$ which satisfy system~\eqref{eq:ode2ndmom}
		with $\sigma=0$, then
		\[
		E\|X(t)\|_{H^{-1}}^2
		\leq E\|X(0)\|_{H^{-1}}^2
		,\qquad\text{for all }t\in[0,T].
		\]
	\end{lemma}
	
	\begin{proof}
		Recall that, by definition, for all $t\geq0$,
		$E\|X(t)\|_{H^{-1}}^2=\sum_{n\geq1}k_n^{-2}u_n(t)$. For $N\geq1$,
		if we sum up to $N$ and differentiate, we have
		\begin{equation*}
		\begin{split}
		\frac d{dt}\sum_{n=1}^Nk_n^{-2n}u_n
		& =\sum_{n=1}^Nk_n^{-2}u_n'\\
		& = k_1^{-2}(-k_1^2u_1+k_1^2u_2) \\
		&\phantom{=}+ \sum_{n=2}^Nk_n^{-2}\big(k_{n-1}^2u_{n-1}-(k_{n-1}^2+k_n^2)u_n+k_n^2u_{n+1}\big)\\
		& = -u_1+u_2+\sum_{n=2}^N\big(\lambda^{-2}u_{n-1}-(\lambda^{-2}+1)u_n+u_{n+1}\big)\\
		& = -(1-\lambda^{-2})u_1-\lambda^{-2}u_N+u_{N+1}
		\leq u_{N+1},
		\end{split}
		\end{equation*}
		hence
		\[
		\sum_{n=1}^Nk_n^{-2n}u_n(t)
		\leq\sum_{n=1}^Nk_n^{-2n}u_n(0)+\int_0^tu_{N+1}(s)\mathrm{d}s.
		\]
		Passing to the limit as $N\to\infty$, the integral converges to zero,
		since $X\in L^2([0,T]\times\Omega;l^2)$, concluding the proof. 
	\end{proof}
	
	We can now prove the uniqueness result for moderate solutions.
	
	\begin{theorem}\label{t:uniqueness_moderate}
		If $X$ and $\widetilde X$ are two moderate solutions with the same
		initial condition $\overline X\in L^2(\Omega, H^{-1})$, defined on
		$[0,T]$ and $[0,\widetilde T]$, with $T\leq\widetilde T$, then
		$X=\widetilde X$ in $[0,T]$ almost surely.
	\end{theorem}
	
	\begin{proof}
		By Definition~\ref{d:moderate_sol}, it is easy to see that
		$X-\widetilde X$ is a moderate solution defined on $[0,T]$ for the
		model with $\sigma=0$ and with zero initial condition, so without
		loss of generality we assume $\sigma=0$, $\overline X=0$ and
		$\widetilde X=0$. Let $X^{(N)}$ and $\overline X^{(N)}$ be as in
		Definition~\ref{d:moderate_sol}. For all $N\geq1$, $X^{(N)}$ is a
		proper solution of the model with $\sigma=0$, so by
		Proposition~\ref{p:sol_2ndmom} we can apply
		Lemma~\ref{l:H-1_contraction}, yielding that for all $t\in[0,T]$
		\[
		E\|X^{(N)}(t)\|_{H^{-1}}^2
		\leq E\|\overline X^{(N)}\|_{H^{-1}}^2.
		\]
		By integrating we get
		\[
		\int_0^TE\|X^{(N)}(t)\|_{H^{-1}}^2\mathrm{d}t
		\leq TE\|\overline X^{(N)}\|_{H^{-1}}^2.
		\]
		Taking the limit for $N\to\infty$, we have that the
		$L^2([0,T]\times\Omega;H^{-1})$-norm of $X$ is zero. Finally, by the
		continuity of trajectories, it is easy to conclude that $X(t)=0$ for
		all $t$, almost surely.
	\end{proof}
	
        \begin{corollary}
          \label{r:uniq_ps}
          Since proper solutions are moderate solutions, uniqueness
          holds in the class of proper solutions too, whatever the
          initial condition.
          This means in particular that the
          inequalities~\eqref{e:ps_en_bd_avg}
          and~\eqref{e:ps_en_bd_as} hold in general for proper solutions with
          initial conditions in $L^4(\Omega;l^2)$ and that the sequence
          of approximants $(X^{(N)})_{N\geq1}$ in
          Definition~\ref{d:moderate_sol} is uniquely determined by
          their initial conditions.
        \end{corollary}

	To conclude this section, we now state and prove the existence result for moderate solutions, using once again Lemma~\ref{l:H-1_contraction}.
	
	\begin{theorem}
		For all $\overline X\in L^2(\Omega;H^{-1})$ there
		exists a moderate solution $X$ with initial condition $\overline
		X$, such that
		\begin{equation}
		\label{e:avg_h1_bd_modsol}
		E\|X(t)\|_{H^{-1}}^2
		\leq 2E\|\overline X\|_{H^{-1}}^2+2\sigma^2t
		,\qquad\text{for all }t\in[0,T].
		\end{equation}
		
		Moreover the approximants $(X^{(N)})_{N\geq1}$ of $X$
                can be taken as the unique proper solutions with the
                following initial conditions
		\begin{equation}\label{e:def_XNn_in_thm_exist_moderate}
		\overline X^{(N)}_n:=-N\vee (\overline X_n\mathbbm1_{n\leq N})\wedge N
		,\qquad n\geq1.
		\end{equation}
	\end{theorem}
	
	\begin{proof}
		By virtue of Theorem~\ref{t:existence_l2init} and
                Corollary~\ref{r:uniq_ps}, there exists a unique proper
		solution $Z$ of~\eqref{e:main_ito} with zero initial condition. Below we will 
		exhibit a
		moderate solution $X$ for the model with $\sigma=0$ and initial
		condition $\overline X$. Then, by linearity, $Z+X$ will be the
		required moderate solution.
		
		We can assume $\sigma=0$. For $N\geq1$, let
		$\overline X^{(N)}\in L^\infty(\Omega,l^2)$ be defined as
		in~\eqref{e:def_XNn_in_thm_exist_moderate}. Then, by
		Theorem~\ref{t:existence_l2init} and
                Corollary~\ref{r:uniq_ps}, there exists a unique proper solution
		$X^{(N)}$ with initial condition $\overline X^{(N)}$. In view of
		Definition~\ref{d:moderate_sol}, we will show that $(X^{(N)})_{N\geq1}$
		is a Cauchy sequence in $L^2([0,T]\times\Omega;H^{-1})$ and that
		$\overline X^{(N)}\to\overline X$ in $L^2(\Omega;H^{-1})$, as
		$N\to\infty$.
		
		For all $M,N\geq1$, the difference $X^{(M)}-X^{(N)}$ is a proper solution,
		hence by Proposition~\ref{p:sol_2ndmom} we can apply
		Lemma~\ref{l:H-1_contraction}, yielding that
		\begin{equation}\label{e:contraction_cauchy_proper_sol}
		\int_0^TE\|X^{(M)}(t)-X^{(N)}(t)\|_{H^{-1}}^2\mathrm{d}t
		\leq TE\|\overline X^{(M)}-\overline X^{(N)}\|_{H^{-1}}^2.
		\end{equation}
		Thus, if we can prove the convergence for the sequence of initial conditions, 
		we also get the Cauchy property for the sequence
		$(X^{(N)})_{N\geq1}$. To this end, consider the measurable functions, on
		$\Omega\times\mathbb N$ defined by
		\[
		\psi(\omega,n):=k_n^{-2}\overline X_n(\omega)^2
		,\ \text{and}\ 
		\psi^{(N)}(\omega,n):=k_n^{-2}\bigl(\overline X^{(N)}_n(\omega)-\overline X_n(\omega)\bigr)^2
		,\  N\geq1.
		\]
		Clearly $\psi^{(N)}\to0$ pointwise, as $N\to\infty$, and also in
		$L^1(\Omega\times\mathbb N)$, since
		$\psi^{(N)}\leq\psi\in L^1(\Omega\times\mathbb N)$. On the other hand we
		have,
		\begin{align*}
		\|\overline X^{(N)}-\overline X\|_{L^2(\Omega;H^{-1})}^2
		& =E\|\overline X^{(N)}-\overline X\|_{H^{-1}}^2\\
		& =E\sum_{n\geq1}k_n^{-2}(\overline X^{(N)}_n-\overline X_n)^2
		=\|\psi^{(N)}\|_{L^1(\Omega\times\mathbb N)}.
		\end{align*}
		Hence $\overline X^{(N)}\to\overline X$ in $L^2(\Omega;H^{-1})$, as
		$N\to\infty$, then by~\eqref{e:contraction_cauchy_proper_sol} the
		sequence $(X^{(N)})_{N\geq1}$ has the Cauchy property and there exists
		the limit $X\in L^2([0,T]\times\Omega;H^{-1})$.
		
		To conclude the proof of the existence statement, we
                need to show that $X$ admits a modification which is a
                componentwise solution, that is, $X$ has a
                modification with continuous adapted trajectories,
                which solves system~\eqref{e:main_ito}.  This is
                completely standard and a simpler version of the
                argument in the proof of
                Theorem~\ref{t:existence_l2init}, with strong $L^2$
                convergence in place of weak $L^2$ convergence.
		
		To prove the bound~\eqref{e:avg_h1_bd_modsol} of the $H^{-1}$ norm,
		we can notice that
		Lemma~\ref{l:H-1_contraction} applies to the approximants
		$X^{(N)}$ and, taking the limit, the same inequality holds for
		$X$. Then it is enough to recall that when $\sigma>0$ we need
		to take the auxiliary proper solution $Z$ into account, for
		which~\eqref{e:ps_en_bd_avg} applies, so that, with 
		the $H^{-1}$ norm controlled by the $l^2$ one,
		\[
		E\|X(t)+Z(t)\|_{H^{-1}}^2
		\leq 2E\|X(t)\|_{H^{-1}}^2 + 2E\|Z(t)\|_{H^{-1}}^2
		\leq 2E\|\overline X\|_{H^{-1}}^2+2\sigma^2t. \qedhere
		\]
	\end{proof}

	\section{Regularity of moderate solutions}
	\label{sec:markovchain}
	Now that we have introduced moderate solutions and shown their existence and uniqueness' results, let us go back to the second moments' 
	system~\eqref{eq:ode2ndmom}, and delve deeper into it.
	We can show a Markov chain associated with our system. This is not surprising, as it is the case for other models in the dyadic family (see for example~\cite{BarFlaMor2011AAP, Bianchi2013, BarMor2013NON, BrzFlaNekZeg}). This associated process will allow us to prove sharper estimates on the norm of solutions, leading us to Theorem~\ref{thm:regularity_moderate} at the end of this section, which states that moderate solutions are, in a sense, much more regular than one would expect from the definition.
	
	Let $\Pi$ be the infinite matrix defined by
	\begin{equation*}
	\Pi_{i,j}=
	\begin{cases}
	-k_1^2 & j=i=1 \\
	-k_{i-1}^2-k_i^2 & j=i\geq2 \\
	k_{i-1}^2 & j=i-1 \\
	k_i^2 & j=i+1
	\end{cases}
	\qquad\text{for }i,j\geq1,
	\end{equation*}
	or, in an equivalent way, 
	\[
	\Pi=\begin{pmatrix}
	-k_1^2 & k_1^2 & 0 & 0 & \dots \\
	k_1^2 & -k_1^2-k_2^2 & k_2^2 & 0 &\dots\\
	0 & k_2^2 & -k_2^2-k_3^2 & k_3^2 & \dots\\
	0 & 0 & k_3^2 & -k_3^2-k_4^2 & \dots \\
	\dots&\dots&\dots&\dots&\dots\\
	\end{pmatrix}.
	\]
	
	With this definition, $\Pi$ is the stable and conservative $q$-matrix
	associated to a continuous-time Markov chain on the positive integers
	(see~\cite{Anderson} for a comprehensive discussion). The
	corresponding Kolmogorov equations are
	\begin{align*}
	u'&=u\Pi
	,\quad u\in L^\infty(\mathbb R_+;l^1(\mathbb R_+))
	\tag{forward}\\
	u'&=\Pi u
	,\quad u\in L^\infty(\mathbb R_+;l^\infty(\mathbb R_+)).
	\tag{backward}
	\end{align*}
	
	Since $\Pi$ is symmetric, both the forward and the backward
	equations are formally equivalent to system~\eqref{eq:ode2ndmom} with
	$\sigma=0$. From now on we will refer in particular to the forward
	equations, because we will be studying the second moments of the
	finite energy solutions of the original system, which will belong to
	the class $L^\infty([0,T];l^1(\mathbb R_+))$.
	
	The forward equations are well-posed. In particular, it is a general fact
	(see for example Theorem 2.2 in~\cite{Anderson} and references
	therein) that, for a $q$-matrix such as $\Pi$, there exists a transition function 
	$f=(f_{i,j}(t))_{i,j\geq1,t\geq0}$ such that, for all $i\geq1$,
	$f_{i,\cdot}$ is a solution of the forward equations with initial
	condition $\delta_{i,\cdot}$, and, for all $j\geq1$, $f_{\cdot,j}$ is a
	solution of the backward equations with initial condition
	$\delta_{\cdot,j}$. This is called the \emph{minimal transition function}
	associated with the $q$-matrix $\Pi$, and has some nice properties, 
	for example $\sum_jf_{i,j}\leq1$ (which is used in the
	proof of Theorem~\ref{t:uniqueness_linear} below). Its uniqueness depends on 
	the form and properties of
	$\Pi$, and in
	our case classical results (see~\cite{Anderson}, again) show that there is uniqueness
	in the class of solutions of the forward equations while there are
	infinitely many solutions in the class of solutions of the backward
	equations.
	
	Nonetheless, we need a statement of uniqueness in a larger class, because we consider $l^1(\mathbb R)$ instead of $l^1(\mathbb R_+)$, and $L^\infty([0,T],l^1)$ instead of $L^\infty(\mathbb R_+, l^1)$. 
	
	\begin{lemma}\label{l:l1_mix_fij}
		Let $\overline u\in l^1$ and $f$ be the minimal transition function
		of $\Pi$. Then
		\[
		u_j(t):=\sum_{i\geq1}\overline u_if_{i,j}(t)
		,\qquad i\geq1
		\]
		defines a solution $u$ of the forward equations in the class
		$L^\infty(\mathbb R_+;l^1)$ with initial condition $\overline{u}$.
	\end{lemma}
	
	\begin{proof}
		Since $f$ is a transition function, it is non-negative and
		$\sum_{j\geq1}f_{i,j}(t)\leq1$ for all $i\geq1$ and all $t\geq0$, so
		in particular,
		\[
		\|u(t)\|_{l^1}
		\leq\sum_{i,j\geq1}|\overline u_i|f_{i,j}(t)
		\leq\|\overline u\|_{l^1}.
		\]
		Formal differentiation gives
		\[
		u'_j(t)
		=\sum_{i\geq1}\overline u_if'_{i,j}(t)
		=\sum_{i,k\geq1}\overline u_if_{i,k}(t)\Pi_{k,j}
		=\sum_{k\geq1}u_k(t)\Pi_{k,j}
		,
		\]
		however, to conclude we must check that differentiation commutes with the sum over
		$i$:
		\[
		\sum_{i,k\geq1}|\overline u_if_{i,k}(t)\Pi_{k,j}|
		\leq\sum_{i\geq1}|\overline u_i|\sum_{k\geq1}|\Pi_{k,j}|
		\leq C(j)\|\overline u\|_{l^1}
		<\infty
		. \qedhere
		\]
	\end{proof}
	
	The following theorem mimics results in~\cite{Anderson}, but requires a new proof nevertheless, as already mentioned, since we are considering different spaces.
	\begin{theorem}\label{t:uniqueness_linear}
		For all $T>0$ there is uniqueness of the solution for the forward equations in the
		class $L^1([0,T];l^1(\mathbb R))$, for any initial
                condition. The same holds for
                system~\eqref{eq:ode2ndmom}, that is, when $\sigma>0$.
	\end{theorem}
	
	\begin{proof}
		By linearity, suppose by contradiction that $u$ is a nonzero
		solution in $L^1([0,T];l^1)$ with null initial
                condition and $\sigma=0$ (this applies to both cases). We start
		by constructing another solution $\tilde u$ defined on the whole 
		$[0,\infty)$. Let $\tau\leq T$ be a time such that $u(\tau)\neq0$ but
		$\|u(\tau)\|_{l^1}<\infty$. Let $\tilde u=u$ on $[0,\tau]$ and extend it
		after $\tau$ with the minimal transition function $f$,
		\[
		\tilde u_j(t+\tau)=\sum_{i\geq1}u_i(\tau)f_{i,j}(t)
		,\qquad j\geq1,\ t\geq0.
		\]
		By Lemma~\ref{l:l1_mix_fij}, $\tilde u$ is a solution of the forward
		equations in the class $L^1([0,T];l^1)\cap L^\infty([T,\infty),l^1)$
		and in particular we can define the residuals
		$r=(r_i(\lambda))_{i\geq1,\lambda>0}$ as
		\[
		r(\lambda):=\int_0^\infty \lambda e^{-\lambda t}\tilde u(t)\mathrm{d}t\in l^1.
		\]
		Then by integrating by parts using $\tilde u(0)=0$,
		\[
		r(\lambda)
		=\int_0^\infty e^{-\lambda t}\tilde u'(t)\mathrm{d}t
		=\int_0^\infty e^{-\lambda t}\tilde u(t)\Pi \mathrm{d}t,
		\]
		we get the algebraic relation
		$\lambda r(\lambda)=r(\lambda)\Pi$, that is
		\[
		\begin{cases}
		\lambda r_1=k_1^2(r_2-r_1)\\
		\lambda r_i=k_i^2(r_{i+1}-r_i)-k_{i-1}^2(r_i-r_{i-1})
		,\qquad i\geq2.
		\end{cases}
		\]
		These can be solved recursively: either $r_i=0$ for all $i\geq1$, or
		$r_i/r_1>1$ for all $i\geq2$. 
		To quickly see this, one can prove by
		induction on $i$ that $r_i/r_1>r_{i-1}/r_1\geq1$. The base case for
		$i=2$ comes from the first equation, while the inductive step comes 
		from the second one:
		\[
		k_i^2\biggl(\frac{r_{i+1}}{r_1}-\frac{r_i}{r_1}\biggr)
		=\lambda\frac{r_i}{r_1}+k_{i-1}^2\biggl(\frac{r_i}{r_1}-\frac{r_{i-1}}{r_1}\biggr)
		>0.
		\]
		We had $r(\lambda)\in l^1$, so $r(\lambda)=0$ for all $\lambda>0$
		yielding $\tilde u=0$ and hence a contradiction.
	\end{proof}
	
	\begin{remark}
		With this proof, $l^1$ is the best space we can get: if we relax to $l^{1^-}$, we do not get the contradiction, since the $r_i$'s might not explode, and one can actually show that 
		\[
		k_i^2\biggl(\frac{r_{i+1}}{r_1}-\frac{r_i}{r_1}\biggr)
		\]
		converges to a constant.
	\end{remark}
	
	We are now able to characterize the evolution in time of the second moments as a
	transformation through $\Pi$ of the second moments at time $0$.
	
	\begin{proposition}\label{p:2ndmom_as_pij}
		Let $\sigma=0$ and $f$ be the minimal transition function of
		$\Pi$. If $X$ is the moderate solution with initial condition
		$\overline X\in L^2(\Omega,H^{-1})$, then for all $j\geq1$ and
		$t\in[0,T]$,
		\begin{equation}\label{e:2ndmom_as_pij}
		E[X_j^2(t)]=\sum_{i\geq1}E[\overline X_i^2]f_{i,j}(t)<\infty.
		\end{equation}
	\end{proposition}
	
	\begin{proof}
		For $i\geq1$ and $t\in[0,T]$, let $u_i(t):=E[(X_i(t))^2]$,
		$\overline u_i:=E[\overline X_i^2]$ and
		$v_i(t):=\sum_{h\geq1}\overline u_hf_{h,i}(t)$.
		
		As a first step, we prove the statement in the case that
		$\overline X\in L^2(\Omega;l^2)$ and $X$ is a proper solution.  In
		this case $\overline u\in l^1$, Lemma~\ref{l:l1_mix_fij} applies,
		and so $v$ is a solution of the forward equations in
		$L^\infty(\mathbb R_+;l^1)$. On the other hand, by
		Proposition~\ref{p:sol_2ndmom} and since $\sigma=0$, $u$ is a solution of the forward
		equations in $L^1([0,T];l^1)$. Both have initial condition
		$\overline u$, so by Theorem~\ref{t:uniqueness_linear}, $u=v$ on $[0,T]$.
		
		We turn to the general case of $X$ moderate solution. For $N\geq1$,
		let $X^{(N)}$ and $\overline X^{(N)}$ be approximating sequences. We
		can take the initial conditions $\overline X^{(N)}$ in the form presented
		in~\eqref{e:def_XNn_in_thm_exist_moderate} without loss of
		generality, given that the moderate solution is unique by
		Theorem~\ref{t:uniqueness_moderate}. Let $u_i^{(N)}(t)$,
		$\overline u_i^{(N)}$ and $v_i^{(N)}(t)$ be defined accordingly.
		
		Notice that $\overline X^{(N)}\in L^\infty(\Omega;l^2)$, so, by the first
		step,~\eqref{e:2ndmom_as_pij} holds for the approximants $X^{(N)}$
		and $\overline X^{(N)}$, for $N\geq1$
		\[
		u_j^{(N)}(t)
		=E[(X_j^{(N)}(t))^2]
		=\sum_{i\geq 1}E[(\overline{X}_i^{(N)})^2]f_{i,j}(t)
		=\sum_{i\geq 1}\overline u_i^{(N)}f_{i,j}(t)
		=v_j^{(N)}(t).
		\]
		
		Taking the limit as $N\to\infty$, $\overline u_i^{(N)}$ increases
		monotonically to $\overline u_i$ for all $i\geq1$, hence for all
		$t\in[0,T]$ the right-hand side converges monotonically to
		$v_j(t)$. As for the left-hand side, since $X_j^{(N)}\to X_j$ in
		$L^2([0,T]\times\Omega)$, it is not difficult to verify that
		$u_j^{(N)}\to u_j$ in $L^1(0,T)$. By the
		uniqueness of the limit ($L^1$ and pointwise monotone),
		the identity in~\eqref{e:2ndmom_as_pij} is proved, as well as the finiteness for
		a.e.~$t$.
		Since $u_j$ is bounded uniformly on $[0,T]$
		by~\eqref{e:avg_h1_bd_modsol} and $v^{(N)}_j\leq u_j$, the result
		extends to all $t$. 
	\end{proof}
	
	We can now link back to moderate solutions and their connection with the minimal transition function.
	
	\begin{proposition}
          The second moments of a moderate solution always solve
          system~\eqref{eq:ode2ndmom} componentwise.
	\end{proposition}	
	\begin{proof}
          Let $X$ be a moderate solution with approximating sequence
          $X^{(N)}$, $N\geq1$. Let
          $u_i^{(N)}(t):=E[(X_i^{(N)}(t))^2]$,
          as in the proof of
          Proposition~\ref{p:2ndmom_as_pij}.  We are interested
          in~\eqref{eq:ode2ndmom}, of which $u^{(N)}$ is a
          solution. In integral form we can rewrite them as
		\[\begin{split}
		u^{(N)}_1(t)
		&=u^{(N)}_1(0)+\sigma^2 t-\int_0^tk_1^2u^{(N)}_1(s)\mathrm{d}s+\int_0^tk_1^2u^{(N)}_2(s)\mathrm{d}s\\
		u^{(N)}_j(t)
		&=u^{(N)}_j(0)+\int_0^tk_{j-1}^2u^{(N)}_{j-1}(s)\mathrm{d}s-\int_0^t(k_{j-1}^2+k_j^2)u^{(N)}_j(s)\mathrm{d}s\\
		&\phantom{=}+\int_0^tk_j^2u^{(N)}_{j+1}(s)\mathrm{d}s
		\end{split}
		\]
		for $j\geq2$.
		
		Since $u^{(N)}_j\to u_j$ on $[0,T]$ a.e.~and in $L^1$, we can pass
		to the limit in the left-hand side and inside each of the integrals
		at the right-hand side, yielding that the identity holds for $u_j$
		for a.e.~$t$.  Now the continuity of the trajectories of $X_j$
		together with the bound given by equation~\eqref{e:avg_h1_bd_modsol}
		allows us to conclude that $u_j$ is continuous in $t$. Since the
		right-hand side is also continuous, we can remove the ``a.e.''~and
		the statement is proved. 
	\end{proof}
	
	Proposition~\ref{p:2ndmom_as_pij} is the key tool to control the flow of energy between
	components for moderate solutions and it ensures that nothing
	different happens with respect to proper solutions. Estimates on the minimal
	transition function $f$ will now allow us to compute different norms and
	get regularity results for the moderate solutions.
	
	\subsection{Transition function estimates}
	Let $(S,\mathcal S,\mathcal P)$ be a probability space with a
	continuous-time Markov chain on the positive integers
	$(\xi_t)_{t\geq0}$, with the property of being the minimal process
	associated with $\Pi$, that is
	\[
	f_{i,j}(t)
	=\mathcal P(\xi_t=j|\xi_0=i)
	=:\mathcal P_i(\xi_t=j)
	,\qquad i,j\geq1,\quad t\geq0,
	\]
	where $f$ is the minimal transition function.
	We will not fix the law of $\xi_0$ which will not be relevant, as we
	will be always conditioning on this random variable.
	
	The arguments in the following lemmas are very similar to the ones in Lemmas 10 to
	14 in~\cite{BarFlaMor2011AAP}, but are restated and proven again here, with the right
	generality for this paper and compatible notation.

	\begin{lemma}\label{l:Tj_mean}
		Let $f$ be the minimal transition function and $T_j$ the total time spent by $\xi$ in state $j$, for
		$j\geq1$. Let moreover $\mathcal E_i$ denote the expectation with respect to
		$\mathcal P_i$. Then,
		\begin{equation*}
		\int_0^\infty f_{i,j}(t)\mathrm{d}t
		=\mathcal E_i(T_j)
		=\frac{\lambda^{-2(i\vee j)}}{1-\lambda^{-2}}
		=\frac{k_{i\vee j}^{-2}}{1-k_1^{-2}}.
		\end{equation*}
	\end{lemma}
	
	\begin{proof}
		The first equality is trivial since both terms are equal to
		\[
		\int_0^\infty\mathcal E_i(\mathbbm 1_{\xi_t=j})\mathrm{d}t.
		\]
		We turn to the second one. Let $(\tau_n)_{n\geq0}$ be the jumping times 
		of $\xi$, that is $\tau_0:=0$ and
		\[
		\tau_{n+1}:=\inf\{t>\tau_n:\xi_t\neq\xi_{\tau_n}\}
		,\qquad n\geq0.
		\]
		Let $(\zeta_n)_{n\geq0}$ denote the discrete-time Markov chain
		embedded in $\xi$, that is $\zeta_n:=\xi_{\tau_n}$ for $n\geq0$. For
		every state $j\geq1$, let $V_j$ denote the total number of visits to
		$j$,
		\[
		V_j:=\sum_{n\geq0}\mathbbm1_{\zeta_n=j}.
		\]
		Then, by the strong Markov property and conditioning on the initial state $\xi_0$, $V_j$ is a mixture of a Dirac $\delta_0$ and a geometric random variable.  Specifically, let
		\[
		\pi_{i,j}:=\mathcal P_i(\zeta_n\neq j,\ \forall n\geq0),
		\]
		then
		\[
		V_j(\mathcal
		P_i)=\pi_{i,j}\delta_0+(1-\pi_{i,j})\text{Geom}(\mathcal P_j(\zeta_1=j+1)\pi_{j+1,j}),
		\]
		where we used the fact that the Markov chain is nearest-neighbour
		and that $\pi_{j-1,j}=0$.
		
		For each visit of $\xi$ to site $j$, the time spent there is an exponential random variable with rate $-\Pi_{j,j}$ (recall that $\Pi$ has a negative diagonal). By the strong
		Markov property again, these variables are independent among them and with $V_j$.
		Consequently, we only have to compute
		\[
		\mathcal E_i(T_j)
		=\frac{\mathcal E_i(V_j)}{-\Pi_{j,j}}
		=\frac1{-\Pi_{j,j}}\cdot\frac{1-\pi_{i,j}}{\mathcal P_j(\zeta_1=j+1)\pi_{j+1,j}}
		=\frac1{\Pi_{j,j+1}}\cdot\frac{1-\pi_{i,j}}{\pi_{j+1,j}}.
		\]
		
		Notice that for $j\geq2$,
		\[
		\mathcal P_j(\zeta_1=j+1)
		=\frac{\Pi_{j,j+1}}{-\Pi_{j,j}}
		=\frac{k_j^2}{k_j^2+k_{j-1}^2}
		=\frac{\lambda^{2j}}{\lambda^{2j}+\lambda^{2(j-1)}}
		=\frac1{1+\lambda^{-2}}
		=:\theta>\frac12,
		\]
		independent of $j$, while for $j=1$ the same quantity is 1.  Then
		$\zeta$ is a simple random walk on the positive integers, reflected
		in 1, with positive drift $2\theta-1$. It is now an exercise to
		prove that
		\[
		\pi_{i,j}=
		\begin{cases}
		1-\bigl(\frac{1-\theta}\theta\bigr)^{i-j} & i>j \\
		0 & i\leq j
		\end{cases},
		\]
		that is, $\pi_{i,j}=1-\lambda^{-2[(i-j)\vee0]}$.   Substituting, we can conclude
		\[
		\mathcal E_i(T_j)
		=\frac1{k_j^2}\cdot\frac{\lambda^{-2[(i-j)\vee0]}}{\frac{2\theta-1}\theta}
		=\frac{\lambda^{-2(i\vee j)}}{2-1/\theta}
		=\frac{k_{i\vee j}^{-2}}{1-k_1^{-2}}
		.\qedhere
		\]
	\end{proof}
	
	When the chain starts from $i$, all states $j\geq i$ are visited with
	probability one, and the times $T_j$ have exponential distribution. In
	particular the following holds.
	\begin{corollary}
		For $j\geq1$ the law of $T_j$, conditional on $\xi_0=1$, is
		exponential with mean $\frac{\lambda^{-2j}}{1-\lambda^{-2}}$.
	\end{corollary}
	
	The minimal process $\xi_t$ is uniquely defined up to the time of the
	first infinity (also known as the explosion time),
	$\tau:=\sum_{j\geq1}T_j$ and after that one can assume that it
	rests in an absorbing boundary state $b$ outside the usual state-space
	of the positive integers.
	To estimate the total energy of a solution it will be important to
	deal with $\sum_{j\geq1}f_{i,j}(t)\leq1$ which will be strictly less
	than 1 when there is a positive probability that the chain has reached
	$b$.
	\begin{equation}
	\label{e:massloss_mc}
	\sum_{j\geq1}f_{i,j}(t)
	=\sum_{j\geq1}\mathcal P_i(\xi_t=j)
	=\mathcal P_i\Bigl(\sum_{j\geq1}T_j>t\Bigr)
	=:\mathcal P_i(\tau>t).
	\end{equation}
	
	\begin{lemma}\label{l:exists_t1}
		There exists a time $t>0$ such that $\mathcal{P}_1(\tau>t)<1$.
	\end{lemma}
	
	\begin{proof}
		Let us start the proof by noticing that we have a lower bound on $\mathcal{P}(\tau>t)$ given by $e^{-(\lambda^2-1)^2t}$. 
		We can now introduce the sequence $\vartheta_n = \alpha n k_n^{-2}$, where $\alpha = (\lambda^2-1)^2\lambda^{-2}$ is a constant. Observe that $\sum_{n=1}^{\infty}\vartheta_n=1$. Now we have
		\begin{equation*}
		\begin{split}
		\mathcal{P}(\tau>t) =&\mathcal{P}\Big(\sum_n T_n>\sum_n\vartheta_n t\Big)\\
		\leq & \mathcal{P}\Big(\bigcup_{n=1}^{\infty}\{T_n>\vartheta_n t\}\Big)\\
		\leq& \sum_{n=1}^{\infty}\mathcal{P}(T_n>\vartheta_n t)\\
		=&\sum_{n=1}^{\infty}\exp\Big\{-\frac{1-\lambda^{-2}}{\lambda^{-2n}}\vartheta_n t\Big\}\\
		=&\sum_{n=1}^{\infty}\exp\Big\{-\frac{(\lambda^2-1)^3}{\lambda^4}nt\Big\}\\
		=&\Big(\exp\Big\{\frac{(\lambda^2-1)^3}{\lambda^4}t\Big\}-1\Big)^{-1},
		\end{split}
		\end{equation*}
		and any $t>\log2\cdot \lambda^4(\lambda^2-1)^{-3}$ satisfies our claim. 
	\end{proof}
	
	In the following lemma, we show that it is enough to show the strict inequality $\mathcal{P}(\tau>t)<1$ for just any single time $t$, and then it holds for all positive times.
	
	\begin{lemma}\label{l:ptauless1}
		Assume that there exists a time $\tilde{t}$ such that $\mathcal{P}(\tau>\tilde{t})<1$. Then $\mathcal{P}(\tau>t)<1$ for all $t>0$.
	\end{lemma}
	\begin{proof}
		First of all, we notice that $\mathcal{P}(\tau>t)\leq \mathcal{P}(\tau>s)$, 
		for all $0<s\leq t$: we can read this from the Chapman-Kolmogorov equations,
		\[
		\mathcal{P}(\tau>t) = \sum_{n=1}^{\infty}\mathcal{P}(\tau>t-s|\xi_0=n)\mathcal{P}(\xi_s=n)\leq \sum_{n=1}^{\infty}\mathcal{P}(\xi_s=n) = \mathcal{P}(\tau>s).
		\]
		This tells us that the map $\mathcal{P}(\tau>t)$ is not increasing in $t$, and in particular it is always less than 1 for $t\geq\tilde{t}$.
		Now suppose that there exists a $t<\tilde{t}$ such that $\mathcal{P}(\tau>t)=1$. Then, for any $0<s<t$, 
		\begin{align*}
		1= \mathcal{P}(\tau>t) & = \sum_{n=1}^{\infty}\mathcal{P}(\tau>s|\xi_0=n)\mathcal{P}(\xi_{t-s}=n)\\ 
		&\leq \sum_{n=1}^{\infty}\mathcal{P}(\xi_{t-s}=n) = \mathcal{P}(\tau>t-s),
		\end{align*}
		but the last term is still a probability, so all the terms must be equal to 1. In particular, this means that $\mathcal{P}(\tau>s|\xi_0=n)=1$ for all $n$. 
		
		Finally, we consider 
		\[
		1>\mathcal{P}(\tau>\tilde{t}) = \sum_{n=1}^{\infty}\mathcal{P}(\tau>s|\xi_0=n)\mathcal{P}(\xi_{\tilde{t}-s}=n) = \mathcal{P}(\tau>\tilde{t}-s),
		\]
		and we can keep repeating it to show that for all $t<\tilde{t}$, $\mathcal{P}(\tau>t)<1$. 
	\end{proof}
	
	The following result tells us that by considering processes conditioned to a staring point in 1, we actually took the worst case scenario. The proof is a standard exercise in continuous time Markov chains.
	\begin{lemma}\label{l:l14bfm2011AAP}
		For all $j\geq1$, $\mathcal{P}_j(\tau>t)\leq \mathcal{P}_1(\tau>t)$.
	\end{lemma}
	
	By combining Lemmas~\ref{l:exists_t1}, \ref{l:ptauless1}
	and~\ref{l:l14bfm2011AAP}, we have that
	\begin{equation}
	\label{e:noinf_notsure}
	\mathcal P_i(\tau>t)<1
	,\qquad i\geq1,\quad t>0
	\end{equation}
	
	\subsection{Moderate solutions are finite energy}
	
	We get now to the first application of the results in the previous subsection: we use them in the following
	proposition, to show anomalous dissipation of average energy for
	moderate solutions starting from finite energy initial conditions. The
	latter hypothesis will be dropped afterwards.
	
	\begin{proposition}\label{p:anom_diss}
		Let $\overline{X}\in L^2(\Omega;l^2)$ and let $X$ be the moderate
		solution with initial condition $\overline{X}$. Then $X$ is a finite
		energy solution. Moreover, if $\sigma=0$ and $\overline X\neq0$,
		then for all $t\in(0,T]$, we have
		$\|X(t)\|_{L^2(\Omega;l^2)}<\|\overline{X}\|_{L^2(\Omega;l^2)}$.
	\end{proposition}
	\begin{proof}
		We start from the second statement, so let $\sigma=0$ and fix
		$t>0$. We can rewrite the energy at time $t$ thanks to
		Proposition~\ref{p:2ndmom_as_pij} and
		equations~\eqref{e:2ndmom_as_pij} and~\eqref{e:massloss_mc}, as
		\[
		E[\|X(t)\|^2_{l^2}]
		= \sum_{j\geq1} E[X_j^2(t)]
		=\sum_{j\geq1}\sum_{i\geq1}E[\overline X_i^2]f_{i,j}(t)
		=\sum_{i\geq1}E[\overline X_i^2]\mathcal P_i(\tau>t).
		\]
		Then we can exploit the strict inequalities~\eqref{e:noinf_notsure}
		for all $i\geq1$ to get the result.
		
		Turning to the first statement, by uniqueness and linearity, we can decompose $X$ as the sum of a proper solution with zero initial condition and a moderate solution with zero forcing. Applying what we proved above, bound~\eqref{e:ps_en_bd_avg} and triangle inequality, yields the result.
	\end{proof}
	
	The next result states formally that moderate solutions are
        ``almost'' finite energy solution, in the sense that whatever
        the initial condition, they jump into $l^2$ immediately (in
        fact they jump into $H^{1^-}$).
	
	\begin{theorem}\label{thm:regularity_moderate}
		Let $\sigma=0$. Let $X$ be the moderate solution with initial
		condition $\overline X\in L^2(\Omega;H^\alpha)$, with
		$\alpha\geq-1$. Then $X\in L^2([0,T]\times\Omega;H^\beta)$ for all
		$\beta<\min(1,\alpha+1)$, with norm bounded by a constant depending
		on $\beta$ and the law of $\overline X$, but not on $T$,
		\begin{equation}
		\label{e:Hbeta_const_bd}
		\|X\|_{L^2([0,T]\times\Omega;H^{\beta})}
		\leq C_{\beta,\mathcal L_{\overline X}}
		<\infty.
		\end{equation}
		In particular:
		\begin{enumerate}[i.]
			\item\label{i:a0} for all initial conditions,
			$X\in L^2([0,T]\times\Omega;H^{\beta})$ for all
			$\beta<0$;
			\item\label{i:a1} if $\overline X\in L^2(\Omega;H^{\alpha})$
			for $\alpha>-1$, then $X\in L^2([0,T]\times\Omega;l^2)$;
			\item\label{i:a2} if $\overline X\in L^2(\Omega;l^2)$, then
			$X\in L^2([0,T]\times\Omega;H^\beta)$ for all $\beta<1$;
			\item\label{i:a3} for all initial conditions,
			$X\in L^2([\epsilon,T]\times\Omega;H^{\beta})$ for all
			$\epsilon>0$ and $\beta<1$.
		\end{enumerate}
	\end{theorem}

	\begin{proof}
		By Proposition~\ref{p:2ndmom_as_pij} and Lemma~\ref{l:Tj_mean}, we
		can explicitly compute all the $L^2([0,T]\times\Omega)$-norms of the
		components of $X$. Let us define $u$ component by component as
		$u_i:=E[\overline X_i^2]$ for $i\geq1$, and fix $\alpha$ and $\beta$
		satisfying the hypothesis, in particular $\beta<1$ and
		$\alpha>\beta-1$. Then, recalling that
                $k_j=\lambda^j$, with $\lambda>1$,
		\begin{equation*}
		\begin{split}
		\|X\|_{L^2([0,T]\times\Omega;H^\beta)}^2
		& = \int_0^T\sum_{i,j\geq1}k_j^{2\beta}u_if_{i,j}(t)\mathrm{d}t\\
		& \leq \sum_{i,j\geq1}k_j^{2\beta}u_i\int_0^\infty f_{i,j}(t)\mathrm{d}t\\
		& = C\sum_{j\geq1}\sum_{i\geq j}k_i^{-2}k_j^{2\beta}u_i+C\sum_{j\geq1}\sum_{i=1}^{j-1}k_j^{-2+2\beta}u_i\\
		& = C\sum_{i\geq1}k_i^{-2}u_i\sum_{j=1}^ik_j^{2\beta}+C\sum_{i\geq1}u_i\sum_{j>i}k_j^{-2+2\beta}\\
		& \leq C\sum_{i\geq1}k_i^{-2}u_i\sum_{j=1}^ik_j^{2\beta}+C'\sum_{i\geq1}k_i^{-2+2\beta}u_i.
		\end{split}
		\end{equation*}
		The second infinite sum is equal to
		$\|\overline X\|^2_{L^2(\Omega;H^{\beta-1})}$, which is finite
		because $\beta-1<\alpha$. If $\beta>0$, the finite sum is bounded by a constant times $k_i^{2\beta}$, hence the first infinite sum behaves like the second one. If $\beta\leq0$, the finite sum is bounded by a constant or by $i$, hence the expression is again controlled by finite quantities.
		
		The points~\ref{i:a0},~\ref{i:a1} and~\ref{i:a2} are
		immediately verified by substituting suitable values of
		$\alpha$ and $\beta$.  As for~\ref{i:a3}, it is a trivial
		consequence of the previous ones applied to subsequent time
		intervals. 
	\end{proof}
	
	This important result has two interesting consequences. First, we can recover a 
	similar bound on the $L^2$ norm even if we consider the unforced case 
	(i.e.~$\sigma=0$), however, in this case, the bound depends on $T$, too. 
	Second, we can show that the evolution of the $L^2$ norm is continuous, as soon 
	as we have $t>0$.
	
	\begin{corollary}
		\label{c:regularity_sigma}
		The assumption that $\sigma=0$ can be dropped from
		Theorem~\ref{thm:regularity_moderate}, with the only difference that
		\begin{equation}
		\label{e:Hbeta_lin_bd}
		\|X\|_{L^2([0,T]\times\Omega;H^{\beta})}^2
		\leq C'_{\beta,\mathcal L_{\overline X}}\cdot(1+T)
		<\infty,
		\end{equation}
		holds instead of~\eqref{e:Hbeta_const_bd}.
	\end{corollary}
	\begin{proof}
          By linearity and uniqueness of moderate solutions, we
          decompose the solution as $X=Y+Z$, where $Z$ has zero
          forcing and $Y$ has constant second moments of
          components. To this end, let $(s_n)_n$ be as in the
          statement of Proposition~\ref{prop:uniq_stationary_2mom},
          and let $Y$ be the unique proper solution with forcing
          $\sigma$ and deterministic initial condition $\overline Y$
          defined by $\overline Y_n:=\sqrt{s_n}$.  By
          Proposition~\ref{p:sol_2ndmom}, the second moments of $Y$
          satisfy system~\eqref{eq:ode2ndmom}, so by uniqueness and
          Proposition~\ref{prop:uniq_stationary_2mom}, the second
          moments of the components of $Y$ are constant, and thus
          $Y(t)\in L^2(\Omega;H^s)$ for all $t\geq0$ and all $s<1$. By
          hypothesis $\overline X\in L^2(\Omega;H^\alpha)$. Then
          $\overline Z:=\overline X-\overline Y\in L^2(\Omega;H^r)$
          for all $r\leq\alpha$, $r<1$. Let $Z$ be the moderate
          solution with no forcing and with initial condition
          $\overline Z$, to which
          Theorem~\ref{thm:regularity_moderate} applies. Then $X=Y+Z$
          has the same regularity as $Z$, and if $\beta$ is like in
          the statement of that theorem,
		\[
		\begin{split}
		\|X\|_{L^2([0,T]\times\Omega;H^{\beta})}^2
		&\leq2\|Y\|_{L^2([0,T]\times\Omega;H^{\beta})}^2+2\|Z\|_{L^2([0,T]\times\Omega;H^{\beta})}^2\\
		&\leq2T\|\overline Y\|_{L^2(\Omega;H^{\beta})}^2+2C_{\beta,\mathcal L_{\overline Z}}
		.\qedhere
		\end{split}
		\]
	\end{proof}
	
	\begin{corollary}
		\label{c:energy_finite_all_t} 
		For any moderate solution $X$ and for all $s<1$ the
		$L^2(\Omega;H^s)$-norm of $X$ is finite and continuous on $(0,T]$.
		In particular, $X(t)\in l^2$ a.s.~for all positive $t$.
	\end{corollary}
	\begin{proof}
		Fix $s<1$, and let
		$\|\cdot\|$ denote the $L^2(\Omega;H^s)$-norm. By the last statement
		of Theorem~\ref{thm:regularity_moderate}, we know that $\|X(t)\|$ is
		finite for a.e.~$t\in(0,T]$. Let $(t_n)_n$ be a sequence of such
		times converging to some $t$, and suppose by contradiction that
		$\lim_n\|X(t_n)\|$ does not exists or is different from $\|X(t)\|$,
		which may or may not be finite. Then without loss of generality we
		can deduce that there exists a subsequence $(n_k)_k$ and a real
		$a$, such that
		\[
		\limsup_k\|X(t_{n_k})\|<a< \|X(t)\|.
		\]
		Then there exists $j_0$ such that
		\[
		\sum_{j=1}^{j_0}k_j^{2s}E[X_j^2(t)]>a^2.
		\]
		The left-hand side is a finite sum of seconds moments, hence it is
		continuous in $t$ by Proposition~\ref{p:2ndmom_as_pij}, yielding
		that for $k$ large also $\|X(t_{n_k})\|>a$, which is a
		contradiction. 
	\end{proof}
	
	\section{Invariant measure}
	\label{sec:invariant}
	This final section deals with invariant measures for the transition
	semigroup associated with moderate solutions. We prove that there exists
	one with support on $H^{1^-}\subset l^2$ which is the unique one among those with
	support on $H^{-1}$.
	
	Let $(P_t)_{t\geq0}$ be the transition semigroup associated to the
	moderate solutions, meaning that
	for all $A\subset H^{-1}$ measurable, $x\in H^{-1}$,
	$\varphi\in C_b(H^{-1})$ and $t\geq0$,
	we define
	\[
	P_t(x,A)
	:=P(X^x(t)\in A) 
	,\qquad\text{and}\qquad
	P_t \varphi(x)
	:= E[\varphi(X^x(t))].
	\]
	where $X^x$ is the moderate solution with deterministic initial
	condition $\overline X=x$. (Notice that we are not specifying $T$: the solution can be taken on any interval $[0,T]$, with $T\geq t$, and the semigroup is well-defined thanks to
	Theorem~\ref{t:uniqueness_moderate}.)
	
	\begin{theorem}
		\label{t:invmeas_exist}
		The semigroup $P_t$ associated to moderate solutions, admits an
		invariant measure supported on $l^2$.
	\end{theorem}
	
	\begin{proof}
		By Corollary~\ref{c:energy_finite_all_t}, $P_t(x,l^2)=1$ for all
		$t>0$ and $x\in H^{-1}$, so it makes sense to consider the semigroup
		restricted to $l^2$.
		
		To prove existence, we rely on Corollary~3.1.2
		in~\cite{DapZab1996ergodicity}, which states that there exists an
		invariant measure for a Feller Markov semigroup $P_t$. This holds under the
		assumption that for some probability measure $\nu$ and sequence of
		times $T_n\uparrow\infty$, the sequence $(R^*_{T_n}\nu)_{n\geq1}$ is
		tight, where $R^*_t$ is the operator on probability measures associated to $P_t$,
		defined by
		\[
		R^*_t\nu(A)
		:=\frac1t\int_0^t\int_{l^2}P_s(x,A)\nu(\mathrm dx)\mathrm ds,
		\]
		for every probability measure $\nu$ on $l^2$ and measurable set $A$ of $l^2$.
		
		Let us start with the tightness. Choose $\nu=\delta_0$ and let
		$\beta\in(0,1)$. The compact set to verify tightness will be the
		$H^\beta$-norm closed ball of radius $r$, which is compact under the
		$l^2$ norm,
		\[
		B(r):=\{x\in l^2:\|x\|_{H^\beta}\leq r\}.
		\]
		Then, for all $T>0$, 
		\[
		\begin{split}
		R^*_T\nu(B(r))
		&=\frac1T\int_0^TP_t(0,B(r))\mathrm{d}t
		=\frac1T\int_0^TP(\|X^0(t)\|_{H^\beta}\leq r)\mathrm{d}t\\
		&\geq1-\frac1T\int_0^Tr^{-2}E\|X^0(t)\|_{H^\beta}^2\mathrm{d}t
		=1-\frac1Tr^{-2}\|X^0\|_{L^2([0,T]\times\Omega;H^\beta)}^2.
		\end{split}
		\]
		Now Corollary~\ref{c:regularity_sigma} applies, and
		by~\eqref{e:Hbeta_lin_bd} there exists a constant $C$ such
		that $R^*_T\nu(B(r))\geq 1-C r^{-2}$ for all $T$ and all $r$, which
		proves the tightness.
		
		Let us now move on to the Feller property: to show that it
		holds, we follow an argument similar to the one hinted to
		in~\cite{BrzFlaNekZeg}.  For $x\in l^2$ and
		$\sigma\in\mathbb R$, let $X^{x,\sigma}$ denote the unique
		moderate solution with deterministic initial condition
		$\overline X=x$ and forcing $\sigma$. Then, if $x$ and $y$ are
		two points il $l^2$, we have,
		\begin{equation}
		\label{eq:L2bound}
		E  [\| X^{x,\sigma} (t) - X^{y,\sigma} (t) \|^2_{l^2}]
		=E  [\| X^{x - y,0} (t)\|^2_{l^2}] \leqslant \| x - y
		\|_{l^2}^2
		,\qquad t\geq0,
		\end{equation}
		where we used uniqueness, linearity (whence forcing terms cancel out)
		and Proposition~\ref{p:anom_diss}.
		
		Now consider a sequence $x_n \rightarrow x$ in $l^2$.
		By equation~\eqref{eq:L2bound}, $X^{x_n,\sigma}\longrightarrow
		X^{x,\sigma}$ in $L^2(\Omega;l^2)$, hence in probability and
		in law, meaning that for all $\varphi \in
		C_b (l^2)$:
		
		\[
		P_t \varphi (x_n) =E  [\varphi (X^{x_n,\sigma} (t))] \rightarrow
		E  [\varphi (X^{x,\sigma} (t))] = P_t \varphi (x), 
		\]
		which gives us the continuity of the semigroup $P_t$. 
	\end{proof}
	
	\begin{remark}
		Theorem~\ref{t:invmeas_exist} can be improved to $H^{1^-}$
		regularity.  In fact, again by
		Corollary~\ref{c:energy_finite_all_t}, $P_t(x,H^s)=1$ for all $s<1$,
		$t>0$ and $x\in H^{-1}$, so actually the invariant measure has
		support on $H^{1^-}:=\bigcap_{s<1}H^s$.
	\end{remark}
	
	To prove the uniqueness of the invariant measure, we use the strategy shown in~\cite{AndBarColForPro2016N}: we formulate the problem as a Kantorovich problem in transport of mass (see for example~\cite{ambrosio2013user, AmbGigSav2008gradient, rachev1998mass}) and proceed by showing a contradiction caused by assuming the existence of two different invariant measures.
	
	\begin{theorem}
		There is a unique invariant measure supported on $H^{-1}$ for the semigroup associated with moderate solutions.
	\end{theorem}
	
	\begin{proof}
		Let us assume, by contradiction, that there are two different invariant measures 
		$\mu^1$ and $\mu^2$.
		We can define the set $\Gamma=\Gamma(\mu^1, \mu^2)$ of admissible transport plans 
		$\gamma$ from $\mu^1$ to $\mu^2$, that is the set of joint measures which have 
		the $\mu^i$ as marginals.
		
		We can also define the functional $\Phi$ on $\Gamma$ in the following way: for $\gamma \in \Gamma$
		\[
		\Phi(\gamma) = \int_{l^2\times l^2} \|x-y\|^2_{l^2} \mathrm{d}\gamma(x,y), 
		\]
		that is, we take as cost function $c(x,y) =
		\|x-y\|^2_{l^2}$. We claim that there exists an optimal
		transport map in Kantorovich problem, that is a
		$\gamma_0\in \Gamma$ such that
		$\Phi(\gamma_0)\leq \Phi(\gamma)$ for all $\gamma\in \Gamma$.
		
		Then we can consider the random vector
		$(\overline{X}^{(1)},\overline{X}^{(2)})$, with joint
		distribution $\gamma_0$, and in particular marginal
		distributions $\mu^1$ and $\mu^2$.  Let $X^{(1)}$ and
		$X^{(2)}$ be the moderate solutions with initial conditions
		$\overline{X}^{(1)}$ and $\overline{X}^{(2)}$,
		respectively. Since $\mu^1$ and $\mu^2$ are invariant
		measures, for each $t>0$ the random vector
		$(X^{(1)}(t), X^{(2)}(t))$ has a joint law
		$\gamma_t\in \Gamma$. Consequently we have
		\begin{equation*}
		\begin{split}
		E\|\overline{X}^{(1)}-\overline{X}^{(2)}\|_{l^2}^2
		&=\int_{l^2\times l^2}\|x-y\|_{l^2}^2\mathrm{d}\gamma_0(x,y)
		= \Phi(\gamma_0)
		\leq \Phi(\gamma_t)\\
		&= \int_{l^2\times l^2}\|x-y\|_{l^2}^2 \mathrm{d}\gamma_t(x,y)
		= E\|X^{(1)}(t)-X^{(2)}(t)\|_{l^2}^2
		\end{split}
		\end{equation*}
		
		To conclude the proof, we need a contradiction, e.g.~a
		contraction property of the solutions.  Let us define
		$X:=X^{(1)}-X^{(2)}$. By linearity and uniqueness, this is the
		moderate solution of the problem with $\sigma=0$ and initial
		condition
		$\overline{X}:=\overline{X}^{(1)}-\overline{X}^{(2)}\in
		L^2(\Omega;H^{-1})$. If
		$\|\overline X\|_{L^2(\Omega;l^2)}=\infty$, the contradiction
		is given by Corollary~\ref{c:energy_finite_all_t}. If instead
		the $l^2$ is finite, Proposition~\ref{p:anom_diss} applies,
		and we have the contraction
		\[
		E(\|X(t)\|_{l^2}^2)<E(\|\overline{X}\|_{l^2}^2)
		,\qquad t>0.
		\]
		
		We are left with the claim. By Theorem 1.5
		in~\cite{ambrosio2013user}, it is enough to check that $c$ is
		lower semicontinuous and bounded from below. To prove the
		former, consider converging sequences $x^{(n)}\to x$ and
		$y^{(n)}\to y$ in $H^{-1}$. If $x-y\not\in l^2$ then
		$\|x^{(n)}-y^{(n)}\|_{l^2}=+\infty$ definitely, because $l^2$ is a
		closed subspace of $H^{-1}$, and otherwise there would be a
		subsequence inside $l^2$ converging to a point outside of
		it. On the other hand, if $x-y\in l^2$, then for all
		$\varepsilon>0$ there exists $k$ such that
		\[
		\sum_{i=1}^k(x-y)_i^2
		\geq \|x-y\|_{l^2}^2-\varepsilon/2.
		\]
		Convergence in $H^{-1}$ implies convergence of components, so
		there exists $n_0$ such that for $n\geq n_0$
		\[
		\sum_{i=1}^k(x^{(n)}-y^{(n)})_i^2
		\geq \sum_{i=1}^k(x-y)_i^2-\varepsilon/2,
		\]
		yielding that $c(x^{(n)},y^{(n)})\geq c(x,y)-\varepsilon$ definitely. 
	\end{proof}
	
	\begin{remark}
		This result only applies to invariant measures for moderate solutions. It is however possible to construct wilder componentwise solutions that are stationary, such as the Gaussian one discussed in~\cite{BrzFlaNekZeg}.
	\end{remark}


\begin{thebibliography}{10}
		
		\bibitem{AleBif2018}
		A.~Alexakis and L.~Biferale.
		\newblock Cascades and transitions in turbulent flows.
		\newblock {\em Physics Reports}, 767-769:1--101, 2018.
		
		\bibitem{ambrosio2013user}
		L.~Ambrosio and N.~Gigli.
		\newblock A user's guide to optimal transport.
		\newblock In {\em Modelling and optimisation of flows on networks}, pages
		1--155. Springer, 2013.
		
		\bibitem{AmbGigSav2008gradient}
		L.~Ambrosio, N.~Gigli, and G.~Savar{\'e}.
		\newblock {\em Gradient flows: in metric spaces and in the space of probability
			measures}.
		\newblock Springer Science \& Business Media, 2008.
		
		\bibitem{Anderson}
		W.~J.~Anderson.
		\newblock {\em Continuous-time {M}arkov chains, an applications-oriented
			approach}.
		\newblock Springer Series in Statistics: Probability and its Applications.
		Springer-Verlag, New York, 1991.
		
		\bibitem{AndBarColForPro2016N}
		L.~Andreis, D.~Barbato, F.~Collet, M.~Formentin, and L.~Provenzano.
		\newblock Strong existence and uniqueness of the stationary distribution for a
		stochastic inviscid dyadic model.
		\newblock {\em Nonlinearity}, 29(3):1156--1169, 2016.
		
		\bibitem{BarBiaFlaMor2013}
		D.~Barbato, L.~A.~Bianchi, F.~Flandoli, and F.~Morandin.
		\newblock A dyadic model on a tree.
		\newblock {\em J. Math. Phys.}, 54(2):021507, 20, 2013.
		
		\bibitem{BarFlaMor2010PAMS}
		D.~Barbato, F.~Flandoli, and F.~Morandin.
		\newblock Uniqueness for a stochastic inviscid dyadic model.
		\newblock {\em Proc. Amer. Math. Soc.}, 138(7):2607--2617, 2010.
		
		\bibitem{BarFlaMor2011AAP}
		D.~Barbato, F.~Flandoli, and F.~Morandin.
		\newblock Anomalous dissipation in a stochastic inviscid dyadic model.
		\newblock {\em Annals of Applied Probability}, 21(6):2424--2446, 2011.
		
		\bibitem{BarMor2013NON}
		D.~Barbato and F.~Morandin.
		\newblock Stochastic inviscid shell models: well-posedness and anomalous
		dissipation.
		\newblock {\em Nonlinearity}, 26(7):1919, 2013.
		
		\bibitem{BarMorRom2011}
		D.~Barbato, F.~Morandin, and M.~Romito.
		\newblock Smooth solutions for the dyadic model.
		\newblock {\em Nonlinearity}, 24(11):3083--3097, 2011.
		
		\bibitem{Bianchi2013}
		L.~A.~Bianchi.
		\newblock Uniqueness for an inviscid stochastic dyadic model on a tree.
		\newblock {\em Electron. Commun. Probab.}, 18:no. 8, 1--12, 2013.
		
		\bibitem{BiaFla2020}
		L.~A.~Bianchi and F.~Flandoli.
		\newblock Stochastic {{Navier}}-{{Stokes Equations and Related Models}}.
		\newblock {\em Milan Journal of Mathematics}, 88(1):225--246, 2020.
		
		\bibitem{BiaMor2017CMP}
		L.~A.~Bianchi and F.~Morandin.
		\newblock Structure {{Function}} and {{Fractal Dissipation}} for an
		{{Intermittent Inviscid Dyadic Model}}.
		\newblock {\em Communications in Mathematical Physics}, 356(1):231--260, 2017.
		
		\bibitem{BrzFlaNekZeg}
		Z.~Brze{\'z}niak, F.~Flandoli, M.~Neklyudov, and B.~Zegarli{\'n}ski.
		\newblock Conservative interacting particles system with anomalous rate of
		ergodicity.
		\newblock {\em Journal of Statistical Physics}, 144(6):1171--1185, 2011.
		
		\bibitem{CheFri09}
		A.~Cheskidov and S.~Friedlander.
		\newblock The vanishing viscosity limit for a dyadic model.
		\newblock {\em Physica D: Nonlinear Phenomena}, 238(8):783--787, 2009.
		
		\bibitem{CheFriPav2007}
		A.~Cheskidov, S.~Friedlander, and N.~Pavlovi{\'c}.
		\newblock Inviscid dyadic model of turbulence: the fixed point and {O}nsager's
		conjecture.
		\newblock {\em J. Math. Phys.}, 48(6):065503, 16, 2007.
		
		\bibitem{DapZab1996ergodicity}
		G.~Da~Prato and J.~Zabczyk.
		\newblock {\em Ergodicity for infinite dimensional systems}, volume 229.
		\newblock Cambridge University Press, 1996.
		
		\bibitem{FriGlaVic2016AIHPPS}
		S.~Friedlander, N.~{Glatt-Holtz}, and V.~Vicol.
		\newblock Inviscid limits for a stochastically forced shell model of turbulent
		flow.
		\newblock {\em Annales de l'Institut Henri Poincar\'e, Probabilit\'es et
			Statistiques}, 52(3):1217--1247, 2016.
		
		\bibitem{FriPav04}
		S.~Friedlander and N.~Pavlovi{\'c}.
		\newblock Blowup in a three-dimensional vector model for the {E}uler equations.
		\newblock {\em Comm. Pure Appl. Math.}, 57(6):705--725, 2004.
		
		\bibitem{HytNeeVerWei2016}
		T.~Hyt{\"o}nen, J.~{\noopsort{neerven}}van Neerven, M.~Veraar, and L.~Weis.
		\newblock {\em Analysis in {{Banach Spaces}}: {{Volume I}}: {{Martingales}} and
			{{Littlewood}}-{{Paley Theory}}}, volume~1 of {\em Ergebnisse Der
			{{Mathematik}} Und Ihrer {{Grenzgebiete}}. 3. {{Folge}} / {{A Series}} of
			{{Modern Surveys}} in {{Mathematics}}}.
		\newblock {Springer International Publishing}, 2016.
		
		\bibitem{KatPav04}
		N.~H.~Katz and N.~Pavlovi{\'c}.
		\newblock Finite time blow-up for a dyadic model of the {E}uler equations.
		\newblock {\em Trans. Amer. Math. Soc.}, 357(2):695--708 (electronic), 2005.
		
		\bibitem{rachev1998mass}
		S.~T.~Rachev and L.~R{\"u}schendorf.
		\newblock {\em Mass Transportation Problems: Volume I: Theory}, volume~1.
		\newblock Springer Science \& Business Media, 1998.
		
		\bibitem{Rom2013}
		M.~Romito.
		\newblock Uniqueness and blow-up for a stochastic viscous dyadic model.
		\newblock {\em Probability Theory and Related Fields}, pages 1--30, 2013.
		
	\end{thebibliography}
\end{document}